\documentclass[12pt,a4paper]{article}

\usepackage{latexsym,amssymb,amsthm}
\usepackage{url,graphics}
\usepackage[dvips]{graphicx,epsfig}

\title{\bf The tree of decomposition of a biconnected graph}
\author{D.\,V.\,Karpov}
\date{}

\begin{document}
\maketitle
\righthyphenmin=2
\renewcommand*{\proofname}{\bf Proof}
\newtheorem{thm}{Theorem}
\newtheorem{lem}{Lemma}
\newtheorem{cor}{Corollary}
\theoremstyle{definition}
\newtheorem{defin}{Definition}
\theoremstyle{remark}
\newtheorem{rem}{\bf Remark}

\def\mmax{\mathop{\rm max}}
\def\q#1.{{\bf #1.}}
\def\P{{\rm Part}}
\def\I{{\rm Int}}
\def\ch{{\rm ch}}
\def\R{{\rm Bound}}
\def\B{{\rm BT}}
\def\N{{\rm N}}
\def\GM{{\cal GM}}

\def\q#1.{{\bf #1.}}

\sloppy
\righthyphenmin=2
\exhyphenpenalty=10000

\newcommand{\ov}{\overline}

\section{Introduction}
We consider undirected graphs without loops and multiple edges and use standard notations.
For a graph $G$ we denote the set of its vertices by  $V(G)$ and the set of its edges by $E(G)$. We use notations~$v(G)$ and~$e(G)$ for the number of vertices and edges of $G$, respectively. 

We denote the {\it degree} of a vertex $x$ in the graph $G$ by $d_G(x)$. We denote the maximal vertex degree of the graph~$G$ by~$\Delta(G)$. 

Let $\N_G(w)$ denote the {\it neighborhood}  of a vertex  $w\in V(G)$ (i.e. the set of all vertices of the graph~$G$, adjacent to~$w$).

We denote by $\chi(G)$ the {\em chromatic number} of the graph $G$, i.e. the minimal number of colors in a proper vertex coloring of the graph $G$.

For a set of vertices $U\subset V(G)$ we denote by $G(U)$ the induced subgraph of the graph $G$ on the set $U$.

Before introducing the results of our paper let us recall the classic notions of {\em block} and {\em cutpoint} of a connected graph and some of their properties.

\begin{defin} 1) 
For any set $R\subset V(G)\cup E(G)$ we denote by~${G-R}$ the graph obtained from~$G$ after deleting all verices and edges of the set~$E$ and all edges incident to vertices of $R$.

2) Let $x,y\in V(G)$, $xy\notin E(G)$. We denote by $G+xy$ the graph obtained from  $G$ after adding the edge~$xy$.
\end{defin}

\begin{defin}
Let $G$ be a connected graph.  A vertex $a\in V(G)$  is called a {\em cutpoint}, if the graph $G-a$ is disconnected.

A {\em block} of the graph $G$ is its maximal up to inclusion subgraph without cutpoints.
\end{defin}%

Blocks and cutpoints are important instruments, that helped to prove a lot of facts in different areas of graph theory.

\begin{defin}
The {\em tree of blocks and cutpoints}  of a graph~$G$ is a bipartite graph $B(G)$. The vertices of the first part correspond to all cutpoints  $a_1,\dots,a_n$ of the graph  $G$, the vertices of the second part correspond to all blocks $B_1,\dots, B_m$ of the graph $G$ (we denote these vertices as the correspondent blocks). The vertices  $a_i$ and $B_j$ are adjacent if and only if $a_i\in V(B_j)$. 
\end{defin}

It is easy to prove, that the tree of blocks and cutpoints is really a tree and all its leaves correspond to blocks (the proofs can be found in~\cite{X} and other books). Just this tree structure helps to use blocks and cutpoints.

In 1966 Tutte~\cite{T} have constructed a tree that describes the structure of relative disposition of 2-vertex cutsets in a biconnected graph.
We present our point of view  to this problem and construct  a {\it tree of decomposition} for a biconnected graph and, in more general case, for a set of pairwise independent $k$-vertex cutsets in a  $k$-connected graph. Our construction is similar to the one of Tutte but we use other instrument to describe the structure --- the notion of a {\em part of decomposition}, developed in~\cite{k02}.
As a result we obtain a tree that has more in common with classic tree of blocks and cutpoints that Tutte's one.

It is important to show that the developed construction is useful.
We  use the tree of decomposition of a biconnected graph   for estimating the chromatic number of a biconnected graph. With the help of our  construction we describe the critical biconnected graphs.
Before formulating our results we recall some basic notations in connectivity theory.

\subsection{Basic notations }

In this paper the {\em connected component} of a graph is the vertex set of  its maximal up to inclusion connected subgraph.

\begin{defin} Let $R\subset V(G)$. 

1) We call $R$  a {\em cutset}, if the graph $G-R$ is disconnected. 
Denote by  $\mathfrak R(G)$ the set of all cutsets of the graph $G$ and by $\mathfrak R_k(G)$ the set of all  $k$-vertex cutsets of $G$.

2) Let $X,Y \subset V(G)$, $X\not\subset R$, $Y\not\subset R$.  We say that $R$ {\it separates} the set~$X$ from~$Y$, if no two vertices~$v_x\in X$ and~$v_y\in Y$ belongs to the same connected component of the graph $G-R$.

3) We say that  $R$ {\em splits} a set~$X \subset V(G)$,  if the set  $X\setminus R$  is not contained in one connected component of the graph~${G-R}$.

4) A graph~$G$ is {\em $k$-connected}, if $v(G)>k$ and  $G$ has no cutset  that  consists of at most $k-1$ vertices.
\end{defin}

Cutpoints of a connected graph defined above are its 1-vertex cutsets.

\begin{defin}
Let $G$ be a $k$-connected graph. We say that  cutsets ${S,T\in\mathfrak R_k(G)}$ are {\em independent}, if~$S$ does not split~$T$ and~$T$ does not split~$S$. Otherwise we say that these cutsets are {\em dependent}.
\end{defin}

Unfortunately, cutsets consisting of  $k\ge 2$ vertices can be dependent. That arises difficulties in studying the structure of $k$-connected graphs for ${k\ge 2}$.  It is proved in~\cite{Hoh,KP} that there are two alternatives for cutsets ${S,T\in \mathfrak R_k(G)}$: 
whether $S$ and $T$ are independent or each of them splits the other. The proof of this fact is simple.

The following notions introduced in~\cite{k02} are useful for description of the relative disposition of cutsets in a graph.

\begin{defin}
Let~$\mathfrak S \subset \mathfrak{R}(G)$.

1) A set  $A\subset V(G)$ is a {\em part of $\mathfrak S$-decomposition}, if no cutset of $\mathfrak S$ splits $A$, but any vertex $b\in V(G) \setminus A$ is separated from $A$ by some cutset of $\mathfrak S$.

We denote the set of all parts of $\mathfrak S$-decomposition of the graph $G$ by $\P(G,\mathfrak S)$.  In the cases when it is clear what graph is decomposed we will write simply $\P(\mathfrak S)$.

2)  Let $A\in \P(\mathfrak S)$.  A vertex   $x\in A$ is an  {\em inner} vertex, if it doesn't belong to any cutset of~$\mathfrak S$.  The set of all inner vertices of the part $A$ is called the {\em interior} of~$A$ and denoted by~$\I(A)$. 

A vertex $x\in A$ is a  {\em boundary} vertex, if it belongs to some cutset of~$\mathfrak S$.  The set of all boundary vertices of the part $A$ is called the {\em boundary} of~$A$ and denoted by~$\R(A)$. 
\end{defin}

Clearly, $A=\I(A)\cup \R(A)$. A proof of the following lemma is simply and can be found in~\cite[theorem 2]{k06}.

\begin{lem}
\label{tbound1}
Let $\mathfrak S \subset \mathfrak R(G)$ and $A\in \P(\mathfrak S)$. Then the following statements hold.

$1)$ A vertex $x\in \I(A)$ is not adjacent to any vertex from~$ V(G)\setminus A$.  The boundary $\R(A)$  consists of all vertices of the part~$A$, that have adjacent vertices in~$ V(G)\setminus A$.

$2)$ If $\I(A)\ne\varnothing$, then $\R(A)$ separates $\I(A)$  from $V(G)\setminus A$.
\end{lem}

Consider as an illustration for this notions a simple and, in the same time, important particular case: the decomposition of a  $k$-connected graph by one  $k$-vertex cutset~$S$. What is  a part~$A\in \P(S)$? It is easy to see that its interior~$\I(A)$ is a connected component of the graph~$G-S$, and $A$ is the union of this component and the cutset~$S$. Hence, the induced subgraph~$G(A)$ is connected and each vertrex of the cutset~$S$ is adjacent to at least one vertex of~$\I(A)$.

Let us return to the case~$k=1$. Cutpoints of a connected graph~$G$ are its cutsets, their union is~$\mathfrak R_1(G)$. The vertex sets of  blocks of the graph~$G$ are parts of~$\P(\mathfrak R_1(G))$.  We formulate some basic properties of blocks and cutpoints in the language of parts of decomposition. The proof can be found in~\cite{X} and other books.

\begin{lem}
\label{b1}
Let $a$ and $b$ be cutpoints of a connected graph~$G$ and $U\in \P(\{a\})$ be the part that contains~$b$. Then the following statements hold. 

$1)$ The vertex~$b$ is a cutpoint of the graph~$G(U)$.

$2)$ Any cutpoint of the graph~$G(U)$ is a cutpoint of the graph~$G$.
\end{lem}

\subsection{Main results}

Pairs of dependent $k$-cutsets do not allow to construct  a tree similar to the tree of blocks and cutpoints, that describes the structure  of disposition of cutsets from~$\mathfrak R_k(G)$ and parts of $\P(G,\mathfrak S)$. However, in what follows we construct such structure for any subset of~$\mathfrak R_k(G)$ that consists of pairwise independent cutsets. A particular case of this construction is the tree of decomposition of a biconnected graph.

The following results will show analogy between classic blocks of a connected graph and parts of decomposition of a biconnected graphs.

A disconnected graph is planar if and only if every  subgraph induced on its connected components is planar. Clearly, a connected graph is planar if and only if every its block is planar. 
In 1937 Maclane~\cite{Mac} has studied the process of splitting the graph into {\it atoms} and has shown with the help of Kuratowski's theorem that  a biconnected graph is planar if and only if every its atom is planar. We will show the connection between parts of a biconnected graph and its atoms and reformulate the MacLane's theorem in our terms. 

Clearly, the chromatic number of a connected graph is equal to the maximum of chromatic numbers of its biconnected blocks. In the section~\ref{chr}, we prove some upper bounds on the chromatic number of a biconnected graph in terms of chromatic numbers of  subgraphs induced on the parts of decomposition of this biconnected graph.

The notion of {\em list colorings} appears  not long ago. Now list colorings of  are popular object of research in graph theory.
Let a {\em list} $L(v)$ of $k$ colors corresponds to each vertex $v\in V(G)$. {\em List coloring} of vertices of $G$ is a proper coloring, such that each vertex  $v\in V(G)$ is colored with a color from its list $L(v)$. The minimal positive integer $k$, such that there is a list coloring of a graph $G$ for  any set of lists of  $k$ colors is called the {\em choice number}  of the graph $G$ and denoted by  $\ch(G)$. Clearly, $\ch(G)\ge \chi(G)$.  For a biconnected graph $G$ we will prove a bound on $\ch(G)$ with the help of its  tree of decomposition.

At the end of the paper we will study critical biconnected graphs.

\goodbreak
\begin{defin}
 A $k$-connected graph~$G$ with  $v(G)\ge k+2$  is called {\it critical}, if for any vertex $x\in V(G)$ the graph $G-x$ is not  $k$-connected.
\end{defin}

Critical  $k$-connected graphs were studied in~\cite{CKL, Ham}. It was proved in~\cite{Ham} that a critical  $k$-connected graph must have at least two vetrices of degree less than~${3k-1\over 2}$.  For biconnected graphs that means the existence of at least two vertices of degree 2. 
With the help of our construction we will prove that a critical biconnected graph on at least 4 vertices must have at least 4 vertices of degree~2. Note, that one can easily prove this fact with the help of Tutte's construction~\cite{T}, but it was not  done. 
Moreover, we will describe the structure of all critical biconnected graphs that have exactly four vertices of degree~2.

\section{The tree of decomposition}

\begin{defin}
Let  $G$  be a $k$-connected graph and $\mathfrak S\subset \mathfrak R_k(G)$ be such that  cutsets of~$\mathfrak S$ are pairwise independent.

1) We  construct the {\em tree of decomposition} $T(G,\mathfrak S)$ in the following way. Vertices of one part of~$T(G,\mathfrak S)$ are the cutsets of~$\mathfrak S$ and  vertices of the other part are the parts of~$\P(\mathfrak S)$. We  denote the vertices of~$T(G,\mathfrak S)$   as the correspondent sets of vertices of the graph~$G$.  The vertices~$S\in \mathfrak S$ and $A\in \P(\mathfrak S)$ are adjacent in $T(G,\mathfrak S)$ if and only if $S\subset A$.

2) We  construct the graph ~$G^{\mathfrak S}$ on the vertex set~$V(G)$ in the following way: we take the  graph~$G$ and for each cutset~$S\in\mathfrak S$  we add all edges that connect pairs of vertices of the set~$S$. 
\end{defin}

The construction of~$T(G,\mathfrak S)$  is similar to the construction of the tree of blocks and cutpoints. The properties of these two trees are also similar.

\begin{thm}
\label{tgs}
Let  $G$  be a $k$-connected graph and $\mathfrak S\subset \mathfrak R_k(G)$ be such that  cutsets of~$\mathfrak S$ are pairwise independent.
Then the following statements hold.

$1)$ $T(G,\mathfrak S)$ is a tree.

$2)$  Let~$S\in \mathfrak S$. Then $d_{T(G,\mathfrak S)}(S)=|\P(S)|$.  Moreover, for each part~$A\in \P(S)$ there is a unique part~$B\in \P(\mathfrak S)$, such that $B\subset A$ and  $B$ is adjacent to~$S$ in~$T(G,\mathfrak S)$. Every leaf of the tree~$T(G,\mathfrak S)$ corresponds to a part of~$\P(\mathfrak S)$.

$3)$ Let~$S\in \mathfrak S$, $B,B'\in \P(\mathfrak S)$. Then the cutset~$S$ separates $B$ from $B'$   in the graph $G$ if and only if  $S$ separates~$B$ from~$B'$ in the tree $T(G,\mathfrak S)$.
\end{thm}

Before the proof of the theorem we prove some properties of the graph~$G^{\mathfrak S}$.

\begin{lem}
 \label{lgs} Let  $G$  be a $k$-connected graph and $\mathfrak S\subset \mathfrak R_k(G)$ be such that  cutsets of~$\mathfrak S$ are pairwise independent.  Then the following statements hold.

$1)$ $\mathfrak S \subset \mathfrak R_k(G^{\mathfrak S})$. Moreover, $\P(G;\mathfrak S)=\P(G^{\mathfrak S};\mathfrak S)$.

$2)$  Let $\mathfrak T \subset \mathfrak S$,  $B\in \P(G;\mathfrak T)$ and $R\in \mathfrak R(G^{\mathfrak S}(B))$. Then $R\in \mathfrak R(G)$.  In particular, the graph~$G^{\mathfrak S}(B)$ is $k$-connected.
\end{lem}

\begin{proof}
1) Consider any cutset $S\in \mathfrak S$. Since the cutsets of $\mathfrak S$ are pairwise independent, no edge of $E(G^{\mathfrak S}) \setminus E(G)$ joins inner vertices of two distinct parts of~$\P(G;S)$. Hence two vertices are separated by a cutset $T\in \mathfrak S$ in the graph~$G$ if and only if they are separated by $T$ in the graph~$G^{\mathfrak S}$. That immediately implies the statements of item 1.

\begin{figure}[!ht]
	\centering
		\includegraphics[width=0.35\columnwidth, keepaspectratio]{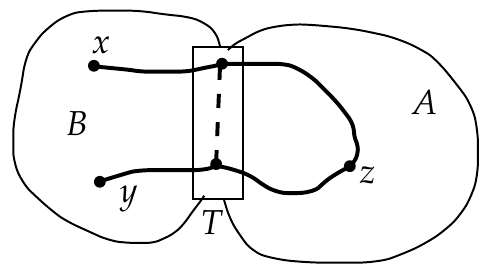}
     \caption{Construction of a  path in $G^{\mathfrak S}(B)$.}
	\label{podp}
\end{figure}

2) Assume that $R\notin \mathfrak{R}(G)$.
Let $x,y\in B$ and~$R$ do not separate $x$ from $y$  in the graph $G$, and, therefore, in $G^{\mathfrak S}$. Consider the shortest $xy$-path $P$ in the graph~$G^{\mathfrak S}-R$. Assume that $P$ contains a vertex~$z\notin B$ (see figure~\ref{podp}). There exists a cutset $T\in \mathfrak T$, that separates $z$ from $B\ni x,y$. Starting at~$z$ and going along the path~$P$ in both directions we reach two vertices $a,b \in T$, these two vertices are adjacent in~$G^{\mathfrak S}$. Hence, there exists  a path, shorter than $P$: one can replace the $ab$-section of the path $P$ by the edge $ab$.  Therefore,  $V(P)\subset B$ and $P$ is a path in~$G^{\mathfrak S}(B)-R$. This  contradicts to the condition of lemma. Hence, $R\in \mathfrak{R}(G)$.

Since $G$ is a $k$-connected graph we have $\mathfrak R_{k-1}(G)=\varnothing$. Hence $\mathfrak R_{k-1}(G^{\mathfrak S}(B))=\varnothing$ and the graph $G^{\mathfrak S}(B)$ is also $k$-connected.
\end{proof}

\renewcommand*{\proofname}{\bf Proof of the theorem~\ref{tgs}}

\begin{proof} We will prove all statements by induction on the number of cutsets in~$\mathfrak S$.  {\sf A $k$-connected graph $G$ is not fixed}.  The {\em base of induction} for empty set $\mathfrak S$ is obvious.

Let us prove the {\it induction step}. Consider the graph~$G'=G^\mathfrak S$. It follows from lemma~\ref{tgs} that the decompositions of the graphs $G$ and $G'$ by the set $\mathfrak S$ coincide, we denote this decomposition by $\P(\mathfrak S)$. Moreover, then $T(G,\mathfrak S)= T(G',\mathfrak S)$. Hence it is enough to prove all statements for the graph~$G'$. 

Let $S\in \mathfrak S$, $\P(S)=\{A_1,\dots, A_n\}$,  $G_i=G'(A_i)$. As we know, all these graphs are $k$-connected. Let the set $\mathfrak S_i$ 
consists of all cutsets of the set~$\mathfrak S$, lying in~$A_i$ and different from~$S$. Then each cutset  from $\mathfrak{S}\setminus S$ belongs to exactly one of the sets $\mathfrak S_1, \dots, \mathfrak S_n$.

\begin{figure}[!hb]
	\centering
		\includegraphics[width=0.7\columnwidth, keepaspectratio]{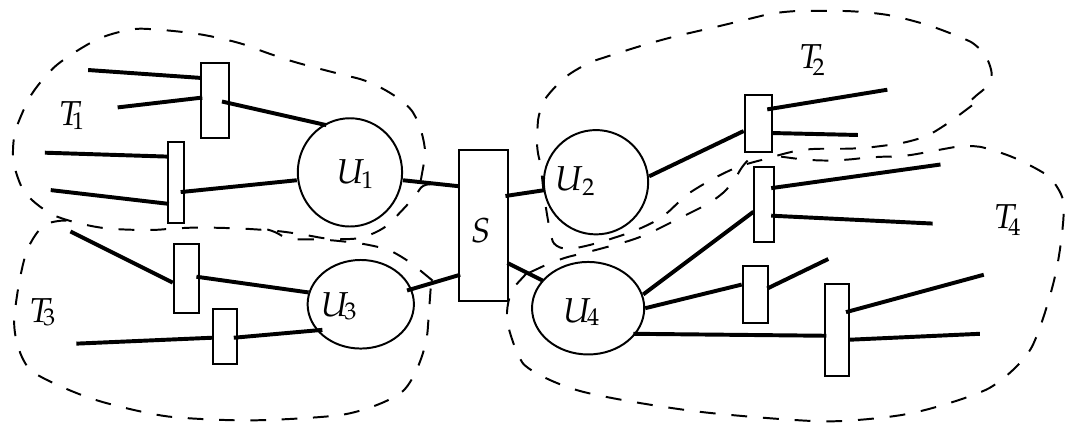}
     \caption{The tree $T(G,\mathfrak S)$.}
	\label{ptree}
\end{figure}

Let $U_i \in \P(G_i;\mathfrak S_i)$ be the part that contains $S$. For each part $U\in \P(G_i;\mathfrak S_i)$ the graph $G'(U)$ is $k$-connected by lemma~\ref{lgs}. Hence, a cutset of the set $\mathfrak S$ that is not contained in $U$ cannot split the graph~$G'(U)$.  
The cutset~$S$ lies in the part~$U_i$, but doesn't split this part since $U_i\subset A_i\in \P(G;S)$. Therefore, we have $\P(G_i;\mathfrak S_i)\subset  \P(G';\mathfrak S)$ and~$U_i$ is the only part of $\P(G_i;\mathfrak S_i)$ that contains~$S$.  Hence,
$$\P(G';\mathfrak S)= \bigcup_{i=1}^n\P(G_i;\mathfrak S_i),$$
this union is disjoint and the parts of~$\P(\mathfrak S)$ that contain the cutset~$S$ are~$U_1,\dots, U_n$. Thus the statement~2 of the theorem is proved for the cutset~$S$ and, similarly, for all other cutsets from~$\mathfrak S$.

Each part of~$\P(G_i;\mathfrak S_i)$, except~$U_i$, is adjacent in  $T_i=T(G_i,\mathfrak S_i)$ and in $T(G,\mathfrak S)$ to the same cutsets. In~$T(G,\mathfrak S)$,  for each part~$U_i$  the edge joining $U_i$ with $S$ is added. Hence $T(G,\mathfrak S)-S$ is a union of exactly  $n$ connected graphs: the graphs $T_i$ (where~$i\in \{1,\dots, n\}$, see figure~\ref{ptree}). By induction assumption all these graphs are trees, hence, the statements of items 1 and 3 of the theorem proved.
\end{proof}

\renewcommand*{\proofname}{\bf Proof}
As we see now, the properties of the tree of decomposition are similar to the well known properties of the classic tree of blocks and cutpoints.

\section{The tree of decomposition of a biconnected graph}

In what follows let $G$ be a biconnected graph. We consider cutsets of the set~$\mathfrak R_2(G)$.

\begin{defin}
The cutset~$S\in \mathfrak R_2(G)$ is called {\em single}, if it is independent with all other cutsets of the set~$\mathfrak R_2(G)$.
Denote by $\mathfrak O(G)$ the set consisting of all single cutsets of~$G$.
\end{defin}

In 1966 Tutte~\cite{T}  described the structure of relative disposition of 2-vertex cutsets in a biconnected graph with the help of a tree. This tree is quite similar to the tree of decomposition of a biconnected graph by the set of all its single cutsets. However, the sets and the tree were defined in~\cite{T} in more complicated way.

Clearly, single cutsets are pairwise independent. That allows us to write the following definition.

\begin{defin}
1) The {\em tree of decomposition}~$\B(G)$ of a biconnected graph~$G$ is the tree~$T(G,\mathfrak O(G))$.

2) We will use the notion $\P(G)$ instead of $\P(\mathfrak{O}(G))$ and call parts of this decomposition simply {\em parts of the graph~$G$}.

A part~$A\in \P(G)$ is called {\it terminal}, if it corresponds to a leaf of the graph~$\B(G)$. 
\end{defin}

\begin{rem}
 1) It follows from theorem~\ref{tgs} that~$\B(G)$ is a tree.

 2) If~$A\in \P(G)$  is a terminal part then $\R(A)$ is a single cutset of the graph~$G$.
\end{rem}

\begin{lem}
 \label{lod} Let $S$ be a single cutset of a biconnected graph~$G$ and $x\in S$.  Then the following statements hold.

$1)$ Let~$d_{\B(G)}(S)=d$.  Then $d_{G}(x)\ge d$. If $d_{G}(x)= d$, then two vertices of the cutset~$S$ are not adjacent.

$2)$ $d_{G}(x)\ge 3$.

\end{lem}

\begin{proof}
1) By theorem~\ref{tgs} we have $|\P(S)|=d_{\B(G)}(S)=d$. For each part of $\P(S)$ there is a vertex adjacent to~$x$ in the interior of this part  (otherwise the graph is not biconnected). Hence $d_{G}(x)\ge d$. In the case where  $d_{G}(x)= d$ all vertices adjacent to~$x$ lie in interiors of  parts of~$\P(S)$.

2) Let $d_G(x)=2$. By item  1 then $|\P(S)|=2$ and the vertices of~$S$ are not adjacent. Hence, $\N_G(x)\in\mathfrak{R}_2(G)$ is a cutset dependent with~$S$. We obtain a contradiction.
\end{proof}

Our next aim is to study parts of a biconnected graph.

\begin{defin} For a biconnected graph~$G$ we denote by~$G'$ the graph~$G^{\mathfrak O(G)}$ (i.e. the graph obtained from~$G$ after adding all edges of type~$ab$ where $\{a,b\}\in \mathfrak O(G)$).
\end{defin}

We prove some important properties  of parts of a biconnected graph and single cutsets, that are similar to properties of blocks and cutpoints (see lemma~\ref{b1}). These properties allow us to  ``split''  a biconnected graph by a single cutset.

\begin{lem}
\label{lg'}  For a biconnected graph~$G$ the following statements hold.

$1)$   Let ~$S\in \mathfrak R_2(G)$, $a,b\in V(G)$. Then the cutset~$S$ separates~$a$ from $b$ in the graph~$G$ if and only  if~$S$ separates~$a$ from $b$ in the graph~$G'$. In particular, $\mathfrak R_2(G)=\mathfrak R_2(G')$. 

$2)$  Let~$S \in \mathfrak R_2(G)$ be not single and $S\subset A\in \P(G)$. Then  $S\in \mathfrak R_2(G'(A))$ and $S$ is not a single cutset in~$G'(A)$.
\end{lem}

\begin{proof} 1) While constructing  $G'$ we add edges joining pairs of vertices that form a single cutset, these pairs of vertices are not separated from each other by any cutset of~$\mathfrak R_2(G)$. This immediately implies the statements of item 1.

2) Let  $S'\in \mathfrak R_2(G)$ be a cutset dependent with~$S$. By item 1 we have~$S,S'\in \mathfrak R_2(G')$ and these two cutsets are dependent in the graph~$G'$.  Since the graph~$G'(A)$ is biconnected,  it is impossible to split the set $S\subset A$ in  $G'$ by deleting less than two vertices from the part~$A$. Hence, $S'\subset A$. Then $S$ and $S'$ split each other in the graph $G'(A)$. Therefore, $S,S'\in \mathfrak R_2(G'(A))$ and these cutsets are dependent. 
\end{proof}

The following lemma characterize non-single cutsets. A similar characterization was used by Tutte~\cite{T}.

\begin{lem}
 \label{lnod1}
Let $S=\{a,b\}\in \mathfrak R_2(G)$ be a non-single cutset. Then  ${|\P(S)|=2}$, for each part  $A\in \P(S)$ the graph $G(A)$ is not biconnected and  has a  cutpoint that separate $a$ from $b$.
\end{lem}

\begin{proof}
 Since~$S$ is non-single, there exists a cutset~$S'\in \mathfrak R_2(G)$ dependent with~$S$. We know that~$S'$ splits $S$. Hence, 
 any  $ab$-path in $G(A)$ intersects~$S'$.  Therefore $S'$ intersects~$\I(A)$.

Thus  $S'$ intersects the interior of each part of~$\P(S)$, hence, ${|\P(S)|=2}$. Moreover, if $\{x\}=S'\cap \I(A)$, then $x$ separates $a$ from $b$ in $G(A)$.
\end{proof}

\begin{thm}
 \label{lcycle}
Let~$G$ be a biconnected graph without single cutsets. Then either  $G$ is triconnected or~$G$ is a simple cycle.
\end{thm}

\begin{rem}
 \label{rcycle}
1) Recall, that a triconnected graph contains at least 4 vertices. In particular, a triangle is not a triconnected graph. Hence, two alternatives of the theorem~\ref{lcycle} are mutually exclusive.

2) The statement of this theorem is a consequence of the results  proved in~\cite{k06} for arbitrary~$k$.  However, we give a simple  proof specially for this theorem. 
\end{rem}

\renewcommand*{\proofname}{\bf Proof of the theorem~\ref{lcycle}}
\begin{proof}
Assume that the graph~$G$ is not triconnected. {\it For each cutset $S=\{a,b\}\in\mathfrak R_2(G)$ and part $A\in \P(S)$ we prove, that $G(A)$ is a simple $ab$-path.}

\begin{figure}[!hb]
	\centering
		\includegraphics[width=0.6\columnwidth, keepaspectratio]{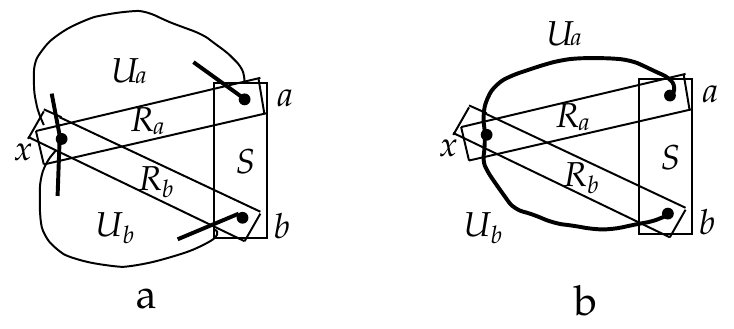}
     \caption{A biconnected graph without single cutsets.}
	\label{pnod}
\end{figure} 

The proof will be induction on $|A|$. The base of induction for the case where the part  $A$ has exactly one inner vertex is obvious. 

The {\it induction step}.  Let the statement be proved for any part less than $A\in \P(S)$. Let $H=G(A)$. Since the cutset~$S$ is non-single by lemma~\ref{lnod1} the graph~$H$ has a cutpoint~$x$ separating  $a$ from $b$. Let~$U_a$ and $U_b$ be connected components of the graph $H-x$, that contain~$a$ and $b$, respectively (see figure~\ref{pnod}a). Since $G$ is biconnected there is no other component in~$H-x$  (any such component would be a connected component in the graph~$G-x$ that is impossible).

Let $U'_a=U_a\setminus \{a\}\ne \varnothing$. Then   $R_a=\{a,x\}$ separates $U'_a$ from other vertices in the graph $G$.  Thus, by induction assumption the graph $G(U'_a\cup R_a)=G(U_a\cup \{x\})$  is a simple  $ax$-path. If~$U_a=\{a\}$, then $\N_H(a)=\{x\}$ and $G(U_a\cup \{x\})$is also a simple $ax$-path.
 Similarly,  $G(U_b\cup \{x\})$ is a simple $bx$-path. Hence the graph   $G(A)$ is a simple  $ab$-path (see figure~\ref{pnod}b). 

\smallskip
Let us finish the proof of the theorem. Let $S=\{a,b\}\in \mathfrak R_2(G)$. 
By lemma~\ref{lnod1} we know, that $\P(S)=\{A_1,A_2\}$.  We have proved that both graphs $G(A_1)$ and $G(A_2)$ are simple $ab$-paths. Hence  $G$ is a simple cycle.
\end{proof}

\renewcommand*{\proofname}{\bf Proof}

\begin{cor}
 \label{cbp}
For each part~$A\in \P(G)$ either the graph~$G'(A)$  is triconnected, or it is a simple cycle.
\end{cor}

\begin{proof}
We know by lemma~\ref{lgs} that the graph~$G'(A)$ is biconnected. Assume that  $S\in \mathfrak R_2(G'(A))$. By lemma~\ref{lgs} we have $S\in \mathfrak R_2(G)$. The cutset~$S$ splits the part $A\in \P(\mathfrak O(G))$, hence, this cutset is non-single. By lemma~\ref{lg'} then $S$ is a non-single cutset in $G'(A)$. 
Hence,  there are no single cutsets in~$G'(A)$. Thus by theorem~\ref{lcycle} either the graph~$G'(A)$  is triconnected, or it is a simple cycle.
\end{proof}

\begin{defin}
Let ~$A \in \P(G)$. The part~$A$ is called a {\it cycle},  if~$G'(A)$ is a simple cycle. The part~$A$ is called a {\em block}, if the graph~$G'(A)$ is triconnected. If the part $A$ is a cycle, then  $|A|$ is called  the {\em length} of the cycle~$A$.
\end{defin}

{\sf Thus we know that any part of a biconnected graph~$G$ is either a cycle or a block.}

\begin{cor}
\label{cint2}
  If a part~$A\in \P(G)$ is a cycle then all vertices of its interior $\I(A)$ have degree~$2$ in the graph~$G$.
\end{cor}

\begin{proof}
 Let $x\in \I(A)$. Then edges can join $x$  in $G$ only to  vertices  of the part~$A$. Clearly, there are exactly two such edges.
\end{proof}

Let us study the disposition of non-single cutsets in the graph~$G$.

\begin{lem}
 \label{lnod}
$1)$ Let $A\in \P(G)$ be a cycle of length at least $4$. Then each pair of non-neighboring vertices of this cycle form a non-single cutset of the graph~$G$.

$2)$ Let  $R\in \mathfrak R_2(G)$  be a non-single cutset of the graph $G$. Then there is a part $A\in \P(G)$, such that $S\subset A$, $A$ is a cycle of length at least~$4$ and $R$ consists of two non-neighboring vertices of this cycle.
\end{lem}

\begin{proof}
 1) Let $A=\{a_1,a_2,\dots, a_k\}$ (the vertices are enumerated in the cyclic order), $R=\{a_1,a_m\}$, where $2<m<k$. Then $R\in \mathfrak R_2(G'(A))$, the cutset $R$ splits   $G'(A)$ into exactly two parts: $U_1=\{a_1,a_2,\dots, a_m\}$ and $U_2=\{a_m,a_{m+1},\dots, a_1\}$.  By lemma~\ref{lgs} we have $R\in\mathfrak R_2(G)$. Clearly, $R\notin \mathfrak O(G)$.

2) The cutset~$R$ is independent with all single cutsets of the graph~$G$, hence there is a part $A\in \P(G)$, such that $S\subset A$. By lemma~\ref{lg'} then $R\in \mathfrak R_2(G'(A))$.  It is lucid from our classification   (see corollary~\ref{cbp}) that then $A$ is a cycle of length at least~4. Now it is clear that $R$ consists of two non-neighboring vertices of this cycle.
\end{proof}

\section{Parts of decomposition and planarity}

Clearly, a connected graph is planar if and only if any its block is planar. In this setion we consider analogous planarity criterion for biconnected graphs --- in terms of parts of this graph.

\begin{defin}
 \label{subdiv}
1) A graph~$H'$ is called a {\it subdivision} of a graph~$H$, if~$H'$ can be obtained from~$H$ after  substituting  some edges by  simple paths. Added vertices of these paths are different, have degree 2 and don't belong to $V(H)$. 
{\it Main} vertices of $H'$ are vertices of the set $V(H)$.

2) 
We denote by $G\supset H$ that the graph $G$ contains a subgraph which is a subdivision of the graph~$H$.
\end{defin}

\begin{lem}
\label{l2podr} 
Let $G$ be a biconnected graph, $A\in \P(G)$. Then $G\supset G'(A)$.
\end{lem}

\begin{proof}
Let $ab\in E(G'(A)) \setminus E(G)$. Then $a,b\in A$ and $\{a,b\}\in \mathfrak O(G)$. Let $U_{a,b}\in \P(\{a,b\})$ be the part that doesn't contain  $A$. Then there exists an $ab$-path $S_{a,b}$ in the graph $G$ which inner vertices belong to  $\I(U_{a,b}$). We substitute the edge $ab$ by the path~$S_{a,b}$.

As a result of all such substitutions we obtain a subgraph~$H$ of the graph~$G$. Let $ab$  and $xy$ be two distinct substituted edges (maybe, they have a common end). Then the parts~$U_{a,b}$ and~$U_{x,y}$ are separated by the part~$A$ in the tree~$\B(G)$, hence, they have no common inner vertex. Therefore, no two added paths has a common vertex. Thus, $H$ is a subdivision of~$G'(A)$. 
\end{proof}

The following theorem almost repeat the theorem proved by MacLane in 1937~\cite{Mac}. 

\begin{thm}
\label{t2plan} 
A biconnected graph~$G$ is planar if and only if for each block $B\in \P(G)$ the graph $G'(B)$ is planar. 
\end{thm}

The only difference of our theorem from MacLane's one is that instead of graphs $G'(B)$ MacLane used so-called {\em atoms}, which, in fact, are subdivisions of graphs~$G'(B)$.
A proof of the theorem~\ref{t2plan} is a simple consequence of well known Kuratowski's theorem on characterization of non-planar graphs.

\section{Parts of decomposition and the chromatic number}
\label{chr}

It is clear, that the chromatic number of a connected graph is equal to the maximum of chromatic numbers of its biconnected blocks. In this section we prove some upper bounds on the chromatic number of a biconnected graph $G$ in terms of   upper bounds on the chromatic numbers of its subgraphs induced on parts of~$G$. These bounds will be easily proved with the help of the tree of decomposition.

\begin{thm}
\label{t2ch} For a biconnected graph~$G$ the following statements hold.

$1)$ 
$$ \chi(G)\le \chi(G')=\mmax_{A\in \P(G)} \chi(G'(A)).
$$

$2)$
\begin{equation}
\label{e2ch} 
\chi(G)\le \mmax_{A\in \P(G)} \chi(G(A))+1.
\end{equation}
\end{thm}

$3)$    $$ \chi(G) \le \max\biggl(3, \quad \mmax_{\mbox{\scriptsize $A$ is a block of $G$}} \chi(G(A))+1 \biggr).
$$

\begin{proof}  Divide  the tree~$\B(G)$ into levels: let level 0 consists of any part~$B\in \P(G)$,  level $\ell+1$ (where $\ell\ge 0$) consists of vertices of~$\B(G)$ that do  not belong to levels $0,\dots, \ell$ and are adjacent to at least one vertex of level 0. It is clear that even levels consist of parts of the graph $G$ and odd levels consist of single cutsets. We will color vertices of parts of $G$ in the order determined by division into levels, starting at level~0.

1) It is enough to color the vertices of $G'$ with    $$ k=\mmax_{A\in \P(G)} \chi(G'(A))$$ colors. Obviously, we can color $G'(B)$ with $k$ colors.   Let vertices of parts that belong to levels  less than $2\ell>0$ are colored. Consider a part $A\in \P(G)$ of level $2\ell$, it is adjacent in $\B(G)$ to exactly one cutset~$S$ of level $2\ell-1$. Vertices of $S$ are the only colored vertices in the part~$A$ and these two vertices have different colors, since they are adjacent in~$G'$. Clearly, there is a proper coloring of the graph~$G'(A)$ with $k$ colors. The vertices of the set $S$ have different colors in this coloring, hence we may color these two vertices  just with the colors they were colored in the coloring of previous levels. 

2) It is enough to color the vertices of  $G$ with    $$ m+1=\mmax_{A\in \P(G)} \chi(G(A))+1$$ colors. We can color the graph $G(B)$  with $m$ colors.  Let vertices of parts that belong to levels  less than $2\ell>0$ are colored. Consider a part $A\in \P(G)$ of level $2\ell$, it is adjacent in $\B(G)$ to exactly one cutset~$S$ of level $2\ell-1$. Vertices of $S=\{a,b\}$ are the only colored vertices in the part~$A$. Let $a$ and $b$ are colored with colors $i$ and $j$ (maybe $i=j$). 

If $i=j$, then we color the vertices of  $G(A)-\{a,b\}$ with  $m$ colors (without the color~$i$). If $i\ne j$,  we color vertices of $G(A)-b$ 
with  $m$ colors (without color $j$) such that  $a$  has color~$i$. In both cases we obtain a proper coloring of vertices of the part~$A$, agreed with the coloring of previous levels.

3) The only difference from item 2  in coloring of a part $A$ is in the case where $A$  is a cycle. Then two vertices of $A$ are colored before and one can easily complete the proper coloring of this cycle using three colors.
\end{proof}

\begin{rem}
 In the proof of statement 2 of theorem~\ref{t2ch} we  can start with coloring of an arbitrary part~$B$, and we need not additional color for this part. Hence, counting the maximum in formula~(\ref{e2ch}) for some part $A\in \P(G)$ we may not increase the chromatic number of the graph $G(A)$ by 1 (just this part must be chosen as~$B$). 

Similarly, in the statement 3 we may not increase by 1 one of the chromatic numbers.
\end{rem}

\begin{cor}
 \label{c2ch}
If all parts of a biconnected graph~$G$ are cycles, then ${\chi(G)\le 3}$.
\end{cor}

We pass to bounds on the choice number of a biconnected graph.

\begin{thm}
\label{t2cho}  For a biconnected graph~$G$ the following statements hold.

$1)$ $$
\ch(G)\le \mmax_{A\in \P(G)} \ch(G(A))+2.
$$
\end{thm}

$2)$    $$ \ch(G) \le \max\biggl(3, \quad \mmax_{\mbox{\scriptsize  is a block of $G$}} \ch(G(A))+2 \biggr).
$$

\begin{proof}
1) Similarly to the proof of theorem~~\ref{t2ch} we divide vertices of the tree~$\B(G)$ into levels and  color vertices of parts of $G$ in the order determined by  levels, starting at level~0.
Let vertices of parts that belong to levels  less than $2\ell>0$ are colored. Consider a part $A\in \P(G)$ of level $2\ell$, it is adjacent in $\B(G)$ to exactly one cutset~$S$ of level $2\ell-1$. Vertices of $S=\{x,y\}$ are the only colored vertices in the part~$A$. 

Let's delete the colors of $x$ and $y$ from the lists of all other vertices of the part $A$. Clearly, the number of remaining colors in these lists is enough for proper coloring of  $G(A\setminus\{x,y\})$.

2) The  difference from item 1 in coloring of a part $A$ is in the case where $A$  is a cycle. Then two vertices of $A$ are colored before. We can easily complete the proper coloring of this cycle:  at the moment we color some vertex~$z$ of this cycle at most two its neighbors are colored and the list of $z$ contains three colors.
\end{proof}

\begin{rem}
 In the proof of statement 2 of theorem~\ref{t2cho} we  can start with coloring of an arbitrary part~$B$, and we need not two additional  colors for this part. Hence, counting the maximum in formula~(\ref{e2ch}) for some part $A\in \P(G)$ we may not increase the choice number of the graph $G(A)$ by 2 (just this part must be chosen as~$B$). 
\end{rem}

\section{Critical biconnected graphs}

The tree of decomposition will help us to study the structure of critical biconnected graphs.

\begin{thm}
 \label{t2cr} 
$1)$ A biconnected graph~$G$ is critical if and only if all its parts-blocks and parts-triangles  have empty interior.

$2)$ Let~$A\in \P(S)$  be a terminal part of a critical biconnected graph~$G$, adjacent in~$\B(G)$ to a single cutset~$S$.  Then~$A$ is a cycle with at list four vertices and all   vertices of~$A$, except two vertices of the cutset~$S$, have degree~$2$ in the graph~$G$.

$3)$ Any critical biconnected graph has at least four vertices of degree $2$.
\end{thm} 

\begin{proof}
1) By lemma~\ref{lnod} vertices not contained in cutsets of $\mathfrak R_2(G)$ (i.e. vertices which deleting does no break biconnectivity of the graph~$G$) are exactly inner vertices of parts-blocks and parts-triangles of~$G$.

\begin{figure}[!ht]
	\centering
		\includegraphics[width=\columnwidth, keepaspectratio]{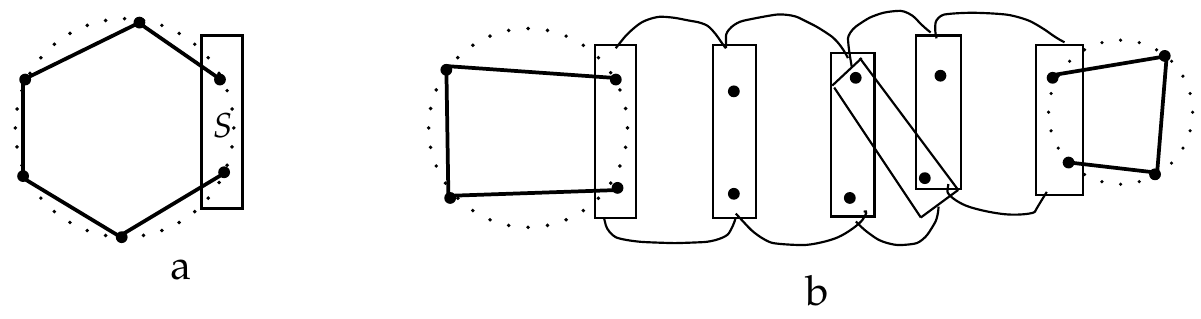}
     \caption{Critical biconnected graphs.}
	\label{part2cr}
\end{figure}

2) Let~$A$ be a terminal part of the graph~$G$. By item~1 then $A$ is a cycle of length $t\ge4$ and~$S$ consists  of two neighboring vertices of this cycle. The interior of~$A$ consists of $t-2\ge 2$  other vertices, by corollary~\ref{cint2} these vertices have degree~2 in~$G$ (see picture~\ref{part2cr}a).

3) If the graph~$G$ has at least one single cutset then it has at least two terminal parts and, by item~2, at least four vertices of degree~2.
Let $G$ has no single cutsets. Clearly, a critical biconnected graph is not triconnected, hence, by theorem~\ref{lcycle} the graph~$G$ is a cycle of length at least~4 and has at least four vertices of degree 2.
\end{proof}

Moreover, now we can describe all critical  biconneceted graphs that have exactly 4 vertices of degree~2. Clearly, a cycle on four verices is the only  such  graph without single cutsets.
Now consider such graph~$G$ which have a single cutset. Then the tree~$\B(G)$ must have exactly two leaves, hence,  all non-terminal parts and all single cutsets have degree two in~$\B(G)$. Therefore, each single cutset splits~$G$ into exactly two parts.

Consider a non-terminal part~$A\in \P(G)$. Since~$d_{\B(G)}(A)=2$, the boundary of~$A$ contains exactly  two single cutsets, thus, $\R(A)$ has 3 or 4 vertices.  Let's prove, that~$\I(A)=\varnothing$. If~$A$ is a block or a triangle, it follows from theorem~\ref{t2cr}. If~$A$ is a cycle of length at least~4, any its inner vertex has degree~2 in~$G$, in this case by theorem~\ref{t2cr} the number of vertices of degree~2 in $G$ is at least 5. 

Thus, a non-terminal part of~$A\in \P(G)$ can be a triangle, a cycle of length 4 or a block on 4 vertices and all vertices of $A$ are contained in two single cutsets, adjacent to $A$ in the tree~$\B(G)$. An example of a critical biconnected graph~$G$ with  4 vertices of degree~2 is shown on  figure~\ref{part2cr}b.


\begin{thebibliography}{99}


\bibitem{Mac} {\sc S.\,MacLane.} {\it A structural characterization of planar combinatorial graphs.}
 Duke Math. J. v.3, Number 3 (1937), p.460-472.


\bibitem{CKL} {\sc G.\,~Chartrand, A.\,~Kaugars and D.\,~R.\,~Lick.}
{\it Critically $n$-connected graphs.}
 Proc. Amer. Math. Soc., v.32 (1972), p.~63-68.


\bibitem{Ham} {\sc Y.\,O.\,Hamidoune.}
   {\it On critically $h$-connected simple graphs.} 
   Discr. Math., 1980, vol.~32, p.~257-262.



\bibitem{T} {\sc W.\,T.\,Tutte.} {\it Connectivity in graphs.}
Toronto, Univ. Toronto Press, 1966.

\bibitem{T2}  {\sc W.\,T.\,Tutte.}
   {\it A theory of $3$-connected graphs.}
   Indag. Math. 1961, vol.~23, p.~441-455.

\bibitem{Hoh}     {\sc W.\,Hohberg.}
{\it The decomposition of graphs into $k$-connected components.}
Discr. Math., {\bf 109}, 1992, p.\,133-145.



\bibitem{X}  {\sc F.\,Harary}, {\it Graph theory,} 1969.

\bibitem{O}  {\sc O.\,Ore}, {\it Theory of graphs,} 1962.

\bibitem{KP} {\sc D.\,V.\,Karpov, A.\,V.\,Pastor.}
{\it O the structure of  $k$-connected graph.}
Zap. nauchn. semin. POMI, {\bf 266}, 2000, p.\,76-106.

\bibitem{k02} {\sc D.\,V.\,Karpov.}
{\it Blocks in $k$-connected graphs.}
Zap. nauchn. semin. POMI, {\bf 293}, 2002, p.\,59-93.


\bibitem{k06} {\sc D.\,V.\,Karpov.}
 {\it Cutsets in a k-connected graph.} 
Zap. nauchn. semin. POMI, {\bf  340}, 2006, p.\,33-60.



\end{thebibliography}
\end{document}